\documentclass[a4paper,centertags,fleqn,reqno,12pt]{amsart}
\usepackage{amsfonts}
\usepackage[left=1.5cm, bottom=3cm]{geometry}
\usepackage{amssymb,amsmath,latexsym}
\usepackage{cite}

\newtheorem{theorem}{Theorem}
\newtheorem{lemma}{Lemma}
\newtheorem{corollary}{Corollary}
\theoremstyle{definition}
\newtheorem{definition}{Definition}
\newtheorem{remark}{Remark}
\numberwithin{equation}{section}

\allowdisplaybreaks
\input{tcilatex}

\begin{document}
\title[On the Chebyshev polynomial and Second Hankel Determinant]{Second Hankel Determinant for certain class of bi-univalent functions defined by Chebyshev polynomials}
\author{H. Orhan, N. Magesh and V. K. Balaji}
\address{Department of Mathematics \\
Faculty of Science, Ataturk University \\
25240 Erzurum, Turkey.\\
\texttt{e-mail:} $orhanhalit607@gmail.com$}
\address{Post-Graduate and Research Department of Mathematics,\\
Government Arts College for Men,\\
Krishnagiri 635001, Tamilnadu, India\\
\texttt{e-mail:} $nmagi\_2000@yahoo.co.in$}
\address{ Department of Mathematics, L. N. Government College, \\
Ponneri, Chennai, Tamilnadu, \,India.\\
\texttt{e-mail:} $balajilsp@yahoo.co.in$}


\begin{abstract}

In this work, we obtain an  upper bound estimate for the second Hankel determinant of a subclass $\mathcal{N}_{\sigma }^{\mu
}\left( \lambda ,t\right) $  of analytic bi-univalent function class  $\sigma$
which is associated with Chebyshev polynomials in the open unit disk. 
\end{abstract}

\maketitle

\section{Introduction and definitions}
Let $\mathcal{A}$ be the class of functions $f$ of the form
\begin{equation}\label{Int-e1}
f(z) = z + \sum\limits_{n=2}^{\infty} a_{n} z^{n},
\end{equation}
which are analytic in the open unit disk $\mathbb{D}=\{z\in\mathbb{C}:\left\vert z\right\vert<1\}.$ We denote by $\mathcal{S}$ the subclass of $\mathcal{A}$ which consists of functions of the form \eqref{Int-e1}, that is, functions which are analytic and univalent in $\mathbb{D}$ and are 
normalized by the conditions $f(0)=0$ and $f'(0)=1$. The Koebe one-quarter
theorem ensures that the image of $\mathbb{D}$ under every
univalent function $f\in\mathcal{S}$ contains the disk with the center in
the origin and the radius $1/4$. Thus, every univalent function $f\in\mathcal{S}$
has an inverse $f^{-1}:f(\mathbb{D})\rightarrow\mathbb{D}$, satisfying
$f^{-1}(f(z))=z$, $z\in\mathbb{D}$, and
\[
f\left(f^{-1}(w)\right)=w \;\; (|w|<r_0(f),\;r_0(f)\geq\frac{1}{4}).
\]

Moreover, it is easy to see that the inverse function has the series expansion
of the form
\begin{equation}\label{expinv}
f^{-1}(w)=w-a_2w^2+\left(2a_2^2-a_3\right)w^3-\left(5a_2^3-5a_2a_3+a_4\right)w^4+\dots\;\; (w\in f(\mathbb{D})).
\end{equation}

A function $f \in \mathcal{A}$ is said to be bi-univalent in $\mathbb{D}$ if
both $f$ and $f^{-1}$ are univalent in $\mathbb{D}.$ Let $\sigma$
denote the class of bi-univalent functions in $\mathbb{D}$ given by (\ref%
{Int-e1}). For a further historical account of functions in the class $\sigma$,
see the work by Srivastava et al. \cite{HMS-AKM-PG}. In fact,
judging by the remarkable flood of papers on non-sharp estimates
on the first two coefficients $a_{2}$ and $a_{3}$ of various subclasses
of the bi-univalent function class $\sigma$ (see, for example,
\cite{Ali-Ravi-Ma-Mina-class,Altinkaya,SA-SY-2016a,SA-SY-2016b,Bulut,MC-ED-HMS:TJM-2017,Caglar-Orhan,Deniz,Deniz-SHD,
BAF-MKA,Jay-NM-JY,HO-NM-VKB,HMS-Caglar,HMS-DB:JEMS-2015,HMS-SSE-RMA:FILO-2015,HMS-SG-FG:AMS-2016,HMS-SG-FG:AM-2017,HMS-SBJ-SJ-HP:PJM-2016,HMS-GMS-NM-GJM,HMS-SSS-RS:TMJ-2014,HT-HMS-SSS-PG:JMI-2016,QHX-YCG-HMS:AML-2012,QHX-HGX-HMS:AMC-2012,Zaprawa,Zaprawa-AAA} and references therein), the above-cited recent pioneering work of Srivastava et al. \cite{HMS-AKM-PG} has apparently revived the study of analytic and bi-univalent functions in recent years.

For functions $f$ and $g,$ analytic in $\mathbb{D} ,$ we say that the
function $f$ is subordinate to $g$ in $\mathbb{D} ,$ and write $f\prec g,$ $z\in \mathbb{D},$
if there exists a Schwarz function $w,$ analytic in $\mathbb{D} ,$ with
$w(0)=0$ and $\left\vert w(z)\right\vert <1$ such that
$f(z)=g(w(z)),$ $z\in \mathbb{D}.$ In particular, if the function $g$ is univalent in $\mathbb{D} ,$ the above
subordination is equivalent to $f(0)=g(0)$ and $f(\mathbb{D} )\subset g(\mathbb{D} ).$

Some of the important and well-investigated subclasses of the univalent
function class $\mathcal{S}$ include (for example) the class $\mathcal{S}%
^{\ast }(\beta )$ of starlike functions of order $\beta $ in $\mathbb{D}$
and the class $\mathcal{K}(\beta )$ of convex functions of order $\beta $
in $\mathbb{D}.$ By definition, we have
\begin{equation*}
\mathcal{S}^{\ast }(\beta ):=\left\{ f:f\in \mathcal{A}\,\text{\ }\mathrm{%
and}\,\Re \left( \frac{zf^{\prime }(z)}{f(z)}\right) >\beta ;\,\,z\in
\mathbb{D};\,\,\,0\leq \beta <1\right\}
\end{equation*}%
and
\begin{equation*}
\mathcal{K}(\beta ):=\left\{ f:f\in \mathcal{A}\,\text{\ }\mathrm{and}\,\Re
\left( 1+\frac{zf^{\prime \prime }(z)}{f^{\prime }(z)}\right) >\beta
;\,\,z\in \mathbb{D};\,\,\,0\leq \beta <1\right\} .  \label{CV-e}
\end{equation*}

For $0\leq \beta <1,$ a function $f\in\sigma$ is in the class $%
\mathcal{S}^*_{\sigma}(\beta)$ of bi-starlike function of order $\beta,$ or $%
\mathcal{K}_{\sigma}(\beta)$ of bi-convex function of order $\beta$ if
both $f$ and $f^{-1}$ are respectively starlike or convex functions of order
$\beta.$


For integers $n \geq 1$ and $q \geq 1,$ the $q-$th Hankel determinant, defined as
\begin{equation*}\label{HD}
H_{q}(n)=\left|\begin{array}{cccc}
                 a_{n} & a_{n+1} & \cdots & a_{n+q-1} \\
                 a_{n+1} & a_{n+2} & \cdots & a_{n+q-2} \\
                 \vdots & \vdots & \vdots & \vdots \\
                 a_{n+q-1} & a_{n+q-2} & \cdots & a_{n+2q-2}
               \end{array}
\right| \qquad (a_1=1).
\end{equation*}
The properties of the Hankel determinants
can be found in \cite{Vein-Dale}. It is interesting to note that
\[
H_{2}(1) = \left|\begin{array}{cc}
                 a_1 & a_2 \\
                 a_2 & a_3
               \end{array}
            \right|
         = a_3- a^2_2 \qquad (a_1=1)
\quad \hbox{and} \quad 
H_{2}(2) = \left|\begin{array}{cc}
                 a_2 & a_3 \\
                 a_3 & a_4
               \end{array}
            \right|
         = a_2a_4- a^2_3 .
\]

The Hankel determinants $H_2(1) = a_3 - a_2^2$ and $H_2(2) = a_2a_4 - a_2^3$
are well-known as Fekete-Szeg\"{o}  and second Hankel determinant functionals
respectively. Further Fekete and Szeg\"{o} \cite{Fekete-Szego} introduced the 
generalized functional $a_3-\delta a_2^2,$ where $\delta$ is some real number. 
In 1969, Keogh and Merkes \cite{Keogh-Merkes} studied the Fekete-Szeg\"{o} 
problem for the classes $\mathcal{S}^*$ and $\mathcal{K}.$ In 2001, Srivastava et al. \cite{HMS-AKM-MKD:CVTP-2001} solved completely the Fekete-Szeg\"o problem for the family $\mathcal{C}_{1}:=\{f\in\mathcal{A}~:~ \Re \left(e^{i\eta}f^{\prime}(z)\right)>0,\; -\frac{\pi}{2} < \eta < \frac{\pi}{2},\; z \in \mathbb{D} \}$ and obtained improvement of $|a_3 - a_2^2|$ for the smaller set $\mathcal{C}_{1}.$  Recently, Kowalczyk et al. \cite{BK-AL-HMS:PIM-2017} discussed the developments involving the Fekete-Szeg\"o functional $|a_3-\delta a_2^2|,$ where $0\leq \delta \leq 1$   as well as the corresponding Hankel determinant for the Taylor-Maclaurin coefficients $\{a_{n}\}_{n\in \mathbb{N}\setminus \{1\}}$ of normalized univalent functions of the form \eqref{Int-e1}. Similarly, several authors
have investigated upper bounds for the Hankel determinant of functions belonging
to various subclasses of univalent functions \cite{NMA-RMA-VR-2017,SHD-Ali-2009,VKD-RRT-SHD-2014,SHD-Lee-2013,GMS-NM-SHD,Orhan-FHD-2010} and the references therein. On the other hand, Zaprawa \cite{%
Zaprawa,Zaprawa-AAA} extended the study on Fekete-Szeg\"{o} problem to some 
specific classes of bi-univalent functions. Following Zaprawa \cite{Zaprawa,Zaprawa-AAA}, 
the Fekete-Szeg\"{o} problem for functions belonging to various subclasses of 
bi-univalent functions were obtained in \cite{Altinkaya,Jay-NM-JY,HO-NM-VKB-Fekete,HT-HMS-SSS-PG:JMI-2016}. Very recently, the upper bounds of $H_2(2)$ for the classes $S^*_{\sigma}(\beta)$ and $K_{\sigma}(\beta)$ were discussed by Deniz et al. \cite{Deniz-SHD}. Later, the upper bounds of $H_2(2)$ for various subclasses of $\sigma$ were obtained by  Alt\i nkaya\ and Yal\c c\i n \cite{SA-SY-2016b,SA-SY-2016c}, \c{C}a\u glar et al. \cite{MC-ED-HMS:TJM-2017}, Kanas et al. \cite{SK-EAA-AZ:2017-mjom} and Orhan et al. \cite{HO-NM-JY-Hankel} (see also \cite{Mustafa,HO-ET-EK:Filo-2017}).

The significance of Chebyshev polynomial in numerical analysis is increased
in both theoretical and practical points of view. Out of four kinds of
Chebyshev polynomials, many researchers dealing with orthogonal polynomials
of Chebyshev. For a brief history of the Chebyshev polynomials of first kind $%
T_{n}(t),$ second kind $U_{n}(t)$ and their applications one can refer
\cite{EHD-1994, JD-RKR-JS-2015, JCM-1967, SA-SY-2016a}. The Chebyshev
polynomials of first and second kinds are well known and they are
defined by%
\begin{equation*}
T_{n}(t)=\cos n\theta \quad \text{and\quad }U_{n}(t)=\frac{\sin (n+1)\theta
}{\sin \theta }\qquad (-1<t<1)
\end{equation*}%
where $n$ denotes the polynomial degree and $t=\cos \theta .$

\begin{definition}
For $\lambda \geq 1,$ $\mu \geq 0$ and $t\in (1/2,1],$ a function $f\in
\sigma $ given by $\left( \ref{Int-e1}\right) $ is said to be in the class $%
\mathcal{N}_{\sigma }^{\mu }\left( \lambda ,t\right) $ if the following
subordinations hold for all $z,w\in \mathbb{D} :$%
\begin{equation}
(1-\lambda )\left( \frac{f(z)}{z}\right) ^{\mu }+\lambda f^{\prime
}(z)\left( \frac{f(z)}{z}\right) ^{\mu -1}\prec H(z,t):=\frac{1}{1-2tz+z^{2}}
\label{Cheby-class-equ-f}
\end{equation}%
and
\begin{equation}
(1-\lambda )\left( \frac{g(w)}{w}\right) ^{\mu }+\lambda g^{\prime
}(w)\left( \frac{g(w)}{w}\right) ^{\mu -1}\prec H(w,t):=\frac{1}{1-2tw+w^{2}}%
,  \label{Cheby-class-equ-g}
\end{equation}%
where the function $g=f^{-1}$ is defined by \eqref{expinv}.
\end{definition}

We note that if $t=\cos \alpha ,$ where $\alpha \in (-\pi /3,\pi /3),$ then
\begin{equation*}
H(z,t)=\frac{1}{1-2\cos \alpha z+z^{2}}=1+\sum\limits_{n=1}^{\infty }\frac{%
\sin (n+1)\alpha }{\sin \alpha }z^{n}\quad (z\in \mathbb{D} ).
\end{equation*}%
Thus
\begin{equation*}
H(z,t)=1+2\cos \alpha z+(3\cos ^{2}\alpha -\sin ^{2}\alpha )z^{2}+\dots
\quad (z\in \mathbb{D} ).
\end{equation*}%
It also can be write
\begin{equation}
H(z,t)=1+U_{1}(t)z+U_{2}(t)z^{2}+\dots \quad (z\in \mathbb{D} ,\quad t\in (-1,1))
\label{CP-SHD-H}
\end{equation}%
where%
\begin{equation*}
U_{n-1}=\frac{\sin (n\, arc \cos t)}{\sqrt{1-t^{2}}}\quad (n\in \mathbb{N})
\end{equation*}%
are the Chebyshev polynomials of the second kind and we have
\begin{equation*}
U_{n}(t)=2tU_{n-1}(t)-U_{n-2}(t),
\end{equation*}%
and
\begin{equation}
U_{1}(t)=2t,\quad U_{2}(t)=4t^{2}-1,\quad U_{3}(t)=8t^{3}-4t,\quad
U_{4}(t)=16t^{4}-12t^{2}+1,\, \dots .  \label{Cheby-Coeffs}
\end{equation}

The generating function of the first kind of Chebyshev polynomial $T_{n}(t),$
$t\in [-1,\, 1]$ is given by
\begin{equation*}
\sum\limits_{n=0}^{\infty }T_{n}(t)z^{n}=\frac{1-tz}{1-2tz+z^{2}}\qquad
(z\in \mathbb{D}).
\end{equation*}

The first kind of Chebyshev polynomial $T_{n}(t)$ and second kind of
Chebyshev polynomial $U_{n}(t)$ are connected by:
\begin{equation*}
\frac{dT_{n}(t)}{dt}=nU_{n-1}(t);\quad T_{n}(t)=U_{n}(t)-tU_{n-1}(t);\quad
2T_{n}(t)=U_{n}(t)-U_{n-2}(t).
\end{equation*}
The class $\mathcal{N}_{\sigma }^{\mu }\left( \lambda ,t\right)$ was introduced and studied by Bulut et al. \cite{BMA}. Also, they discussed initial coefficient estimates and Fekete-Szeg\"{o} bounds for the class $\mathcal{N}_{\sigma }^{\mu }\left( \lambda ,t\right)$ and it's subclasses given in the following remark.
\begin{remark}
For $\mu =1,$ we get the class $\mathcal{N}_{\sigma }^{1}\left(
\lambda ,t\right) =\mathcal{B}_{\sigma }\left( \lambda ,t\right) $ consists
of functions $f\in \sigma $ satisfying the condition%
\begin{equation*}
\left( 1-\lambda \right) \frac{f\left( z\right) }{z}+\lambda f^{\prime
}\left( z\right) \prec H(z,t)=\frac{1}{1-2tz+z^{2}} \qquad
(z\in \mathbb{D})
\end{equation*}%
and%
\begin{equation*}
\left( 1-\lambda \right) \frac{g\left( w\right) }{w}+\lambda g^{\prime
}\left( w\right) \prec H(w,t)=\frac{1}{1-2tw+w^{2}} \qquad
(w\in \mathbb{D})
\end{equation*}%
where the function $g=f^{-1}$ is defined by $\left( \ref{expinv}\right).$
This class was introduced and studied by Bulut et al. \cite{BMB} (see also \cite{Mustafa}).
\end{remark}
\begin{remark}
For $\lambda =1,$ we have a class $\mathcal{N}_{\sigma }^{\mu }\left(
1,t\right) =\mathcal{B}_{\sigma }^{\mu }\left( t\right) $ consists of
bi-Bazilevi\v{c} functions:%
\begin{equation*}
f^{\prime }\left( z\right) \left( \frac{f\left( z\right) }{z}\right) ^{\mu
-1}\prec H(z,t)=\frac{1}{1-2tz+z^{2}} \qquad
(z\in \mathbb{D})
\end{equation*}%
and%
\begin{equation*}
g^{\prime }\left( w\right) \left( \frac{g\left( w\right) }{w}\right) ^{\mu
-1}\prec H(w,t)=\frac{1}{1-2tw+w^{2}} \qquad
(w\in \mathbb{D})
\end{equation*}%
where the function $g=f^{-1}$ is defined by $\left( \ref{expinv}\right)
.$
\end{remark}
\begin{remark}
For $\lambda =1$ and $\mu =1,$ we have the class $%
\mathcal{N}_{\sigma }^{1}\left( 1,t\right) =\mathcal{B}_{\sigma }\left(
t\right) $ consists of functions $f$ satisfying the condition%
\begin{equation*}
f^{\prime }\left( z\right) \prec H(z,t)=\frac{1}{1-2tz+z^{2}} \qquad
(z\in \mathbb{D})
\end{equation*}%
and%
\begin{equation*}
g^{\prime }\left( w\right) \prec H(w,t)=\frac{1}{1-2tw+w^{2}} \qquad
(w\in \mathbb{D})
\end{equation*}%
where the function $g=f^{-1}$ is defined by $\left( \ref{expinv}\right)
.$
\end{remark}
\begin{remark} For $\lambda =1$ and $\mu =0,$ we have the class $%
\mathcal{N}_{\sigma }^{0}\left( 1,t\right) =\mathcal{S}_{\sigma }^{\ast
}\left( t\right) $ consists of functions $f$ satisfying the condition%
\begin{equation*}
\frac{zf^{\prime }(z)}{f(z)}\prec H(z,t)=\frac{1}{1-2tz+z^{2}} \qquad
(z\in \mathbb{D})
\end{equation*}%
and%
\begin{equation*}
\frac{wg^{\prime }\left( w\right) }{g\left( w\right) }\prec H(w,t)=\frac{1}{%
1-2tw+w^{2}} \qquad
(w\in \mathbb{D})
\end{equation*}%
where the function $g=f^{-1}$ is defined by $\left( \ref{expinv}\right).$
\end{remark}
Next we state the following lemmas we shall use to establish the desired bounds in our study.
\par Let $\mathcal{P}$ denote the class of functions $p(z)$ of the form
\begin{equation}\label{Repart-fun-p}
p(z)= 1+ c_{1} z + c_{2} z^{2} + c_{3} z^{3} + \cdots,
\end{equation}
which are analytic in the open unit disc $\mathbb{D}.$
\begin{lemma}\cite{Pom}\label{L-Repart-fun-p}
If the function $p\in \mathcal{P}$ is given by the series (\ref{Repart-fun-p}), then the following sharp estimate holds:
\begin{equation}\label{cnbound}
|c_{k}|\leq 2, \qquad k = 1,\, 2,\, \cdots.
\end{equation}
\end{lemma}
\begin{lemma}\label{L-C2-C3} \cite{Grender}
If the function $p\in \mathcal{P}$ is given by the series (\ref{Repart-fun-p}), then
\begin{eqnarray*}
2c_2 &=& c_1^2 + x (4-c_1^2)\,\, \label{L2-c2}\\
4c_3 &=& c_1^3+2c_1(4-c_1^2)x-c_1(4-c_1^2)x^2+2(4-c_1^2)(1-|x|^2)z \,\, \label{L2-c3}
\end{eqnarray*}
for some $x,$ $z$ with $|x| \leq 1$ and $|z| \leq 1.$
\end{lemma}

In this present paper, we consider a subclass $\mathcal{N}_{\sigma }^{\mu
}\left( \lambda ,t\right) $ of analytic and bi-univalent functions using the
Chebyshev polynomials expansions and find the second Hankel determinant estimates.
Further we discuss its consequences.

\section{Main results}

\begin{theorem}
\label{th-SHD-class} Let $f\in \sigma$ of the form (\ref{Int-e1}) be in $\mathcal{N}_{\sigma }^{\mu }\left( \lambda; \,t\right).$ Then
\[
|a_2a_4-a_3^2| \leq
        \left\{
             \begin{array}{ll}
                  K(2^-, ~t) &; M_{1}\geq0 ~\hbox{and}~ M_{2}\geq0 \\
                  max\left\{\frac{4t^2}{(2\lambda+\mu)^2}, ~K(2^-, ~t)\right\}
                            &; M_{1}>0 ~\hbox{and}~ M_{2}<0  \\
                  \frac{4t^2}{(2\lambda+\mu)^2}
                            &; M_{1}\leq0 ~\hbox{and}~ M_{2}\leq0\\
                  max \left\{K(c_0, ~t),~K(2^-, ~t)\right\}
                            &; M_{1}<0 ~\hbox{and} ~ M_{2}>0
              \end{array}
         \right.,
\]
where
\begin{eqnarray*}
    K(2^-, ~t) & =& \frac{4t^2}{(2\lambda+\mu)^2}
                   +\frac{M_{1}+3M_{2}}{6(\lambda+\mu)^4(2\lambda+\mu)^2(3\lambda+\mu)},
                   \label{CP-SHD-e2a}\\
    K(c_0, ~t) & =& \frac{4t^2}{(2\lambda+\mu)^2}
                   -\frac{3M_{2}^2}{8M_{1}(\lambda+\mu)^4(2\lambda+\mu)^2(3\lambda+\mu)},\qquad c_0 = \sqrt{\frac{-6M_{2}}{M_{1}}}
                   \label{CP-SHD-e2b}
\end{eqnarray*}
and
\begin{eqnarray*}
M_{1}:=M_{1}(\lambda,\,\mu;\,t)
& = & 16t^2 \left|3(2t^{2}-1)(\lambda+\mu)^3-(\mu^2+3\mu+2)(3\lambda+\mu)t^2\right|(2\lambda+\mu)^2
\nonumber\\ && -24t\left[t^2(3\lambda+\mu)+(4t^2-1)(\lambda+\mu)(2\lambda+\mu)\right](\lambda+\mu)^2(2\lambda+\mu)
\nonumber\\ && -24t^2\lambda^2(\lambda+\mu)^3
,
\label{CP-SHD-e3}\\
M_{2}:=M_{2}(\lambda,\,\mu;\,t) &=& 8t\left[t^2(2\lambda+\mu)(3\lambda+\mu)
+(4t^2-1)(\lambda+\mu)(2\lambda+\mu)^2
\right. \nonumber\\ && \left.\qquad +t(2\lambda+\mu)^2(\lambda+\mu) - 2t(\lambda+\mu)^2(3\lambda+\mu)\right](\lambda+\mu)^2.
                            \label{CP-SHD-e4}
\end{eqnarray*}
\end{theorem}
\begin{proof}
Let $f\in \mathcal{N}_{\sigma }^{\mu }\left( \lambda; \,t\right).$ Then
\begin{equation}  \label{SHD-th-p-e1}
(1-\lambda)\left(\frac{f(z)}{z}\right)^{\mu}+\lambda f^{\prime }(z)\left(%
\frac{f(z)}{z}\right)^{\mu-1} = H(t, ~w(z)) \qquad (z \in  \mathbb{D})
\end{equation}
and
\begin{equation}  \label{SHD-th-p-e2}
(1-\lambda)\left(\frac{g(w)}{w}\right)^{\mu}+\lambda g^{\prime }(w)\left(%
\frac{g(w)}{w}\right)^{\mu-1} =  H(t, ~\widetilde{w}(w)) \qquad (w \in  \mathbb{D})
\end{equation}
where $p, ~q \in \mathcal{P}$ and defined by
\begin{equation}
p(z)=\frac{1+w(z)}{1-w(z)}=1+c_1z+c_2z^2+c_3z^3+\dots = 1 + \sum_{n=1}^{\infty}c_n z^n
\label{SHD-th-p-e3}
\end{equation}
and
\begin{equation}
q(w)=\frac{1 + \widetilde{w}(w)}{1-\widetilde{w}(w)} =1+d_1w+d_2w^2+d_3w^3+\dots = 1 + \sum_{n=1}^{\infty}d_n w^n.
\label{SHD-th-p-e4}
\end{equation}
It follows from (\ref{SHD-th-p-e3}) and (\ref{SHD-th-p-e4}) that
\begin{eqnarray}
w(z) & =& \frac{p(z)-1}{p(z)+1}
       = \frac{1}{2}\left[c_1z+\left(c_2-\frac{c_1^2}{2}\right)z^2+\left(c_3-c_1c_2+\frac{c_{1}^{3}}{4}\right)z^3+\ldots\right]
                            \label{CP-SHD-e7a}
\end{eqnarray}
and
\begin{eqnarray}
\widetilde{w}(w) & =& \frac{q(w)-1}{q(w)+1}
                  = \frac{1}{2}\left[d_1w+\left(d_2-\frac{d_{1}^{2}}{2}\right)w^2+\left(d_3-d_1d_2+\frac{d_{1}^{3}}{4}\right)w^3+\ldots\right].
                            \label{CP-SHD-e7b}
\end{eqnarray}
\noindent From \eqref{CP-SHD-e7a} and \eqref{CP-SHD-e7b}, taking $H(z,t)$ as given in \eqref{CP-SHD-H}, we can show that,
\begin{eqnarray}
H(t, ~w(z)) & =& 1+\frac{U_1(t)}{2}c_{1}z
                    +\left[\frac{U_1(t)}{2}\left(c_2-\frac{c_{1}^{2}}{2}\right)+\frac{U_2(t)}{4}c_{1}^{2}\right]z^2\\
                  && \qquad
                    +\left[\frac{U_1(t)}{2}\left(c_3-c_1c_2+\frac{c_{1}^{3}}{4}\right)
                    +\frac{U_2(t)}{2}c_1\left(c_2-\frac{c_{1}^{2}}{2}\right)
                    +\frac{U_3(t)}{8}c_{1}^{3}\right]z^3 + \ldots\nonumber
\label{CP-SHD-e8}
\end{eqnarray}
and
\begin{eqnarray}
    H(t, ~\widetilde{w}(w)) &=& 1+\frac{U_1(t)}{2}d_1w
                              +\left[\frac{U_1(t)}{2}\left(d_2-\frac{d_{1}^{2}}{2}\right)+\frac{U_2(t)}{4}d_{1}^{2}\right]w^2\\
                          &&  \,\,
                              +\left[\frac{U_1(t)}{2}\left(d_3-d_1d_2+\frac{d_{1}^{3}}{4}\right)
                              +\frac{U_2(t)}{2}d_1\left(d_2-\frac{d_{1}^{2}}{2}\right)
                              +\frac{U_3(t)}{8}d_{1}^{3}\right]w^3+\cdots\nonumber
\label{CP-SHD-e9}
\end{eqnarray}
It follows from (\ref{SHD-th-p-e1}), \eqref{CP-SHD-e8} and (\ref{SHD-th-p-e2}) , \eqref{CP-SHD-e9}, we obtain that
\begin{eqnarray}
(\lambda+\mu)a_2 &=&  \frac{U_1(t)}{2}c_1 \label{SHD-th-p-e5}\\
(2\lambda+\mu)\left[a_3+\frac{a_2^2}{2}(\mu-1)\right]&=& \frac{U_{1}(t)}{2}\left(c_2-\frac{c_{1}^{2}}{2}\right)+\frac{U_2(t)}{4}c_{1}^{2} \label{SHD-th-p-e6} \\
(3\lambda+\mu)\left[a_4+(\mu-1)a_2a_3+(\mu-1)(\mu-2)\frac{a_2^3}{6}\right]&=&  \frac{U_1(t)}{2}\left(c_3-c_1c_2+\frac{c_{1}^{3}}{4}\right)\label{SHD-th-p-e7}\\
                   &&\, +\frac{U_2(t)}{2}c_1\left(c_2-\frac{c_{1}^{2}}{2}\right)+\frac{U_3(t)}{8}c_{1}^{3}\nonumber
\end{eqnarray}
and
\begin{eqnarray}
-(\lambda+\mu)a_2 &=&  \frac{U_1(t)}{2}d_1 \label{SHD-th-p-e8}\\
(2\lambda+\mu)\left[(\mu+3)\frac{a_2^2}{2}-a_3\right]&=& \frac{U_{1}(t)}{2}\left(d_2-\frac{d_{1}^{2}}{2}\right)+\frac{U_2(t)}{4}d_{1}^{2} \label{SHD-th-p-e9}\\
(3\lambda+\mu)\left[(4+\mu)a_2a_3-(4+\mu)(5+\mu)\frac{a_2^3}{6}-a_4\right]&=& \frac{U_1(t)}{2}\left(d_3-d_1d_2+\frac{d_{1}^{3}}{4}\right)\label{SHD-th-p-e10}\\
                     &&\,+\frac{U_2(t)}{2}d_1\left(d_2-\frac{d_{1}^{2}}{2}\right)+\frac{U_3(t)}{8}d_{1}^{3}.\nonumber
\end{eqnarray}
From (\ref{SHD-th-p-e5}) and (\ref{SHD-th-p-e8}), we find that
\begin{equation}  \label{SHD-a2}
\frac{U_1(t)}{2(\lambda+\mu)}c_1 = -\frac{U_1(t)}{2(\lambda+\mu)}d_1
\end{equation}
and
\begin{equation}  \label{SHD-th-p-e11}
c_1 =-d_1.
\end{equation}
Next, subtracting (\ref{SHD-th-p-e9}) from (\ref{SHD-th-p-e6}) and using (\ref{SHD-a2}), we arrive at
\begin{equation}  \label{SHD-a3}
    a_3=a_{2}^{2}+\frac{U_1(t)}{4(2\lambda+\mu)}(c_2-d_2)=\frac{U_{1}^{2}(t)}{4(\lambda+\mu)^2} c_{1}^{2}+\frac{U_1(t)}{4(2\lambda+\mu)}(c_2-d_2).
\end{equation}
On the other hand, subtracting \eqref{SHD-th-p-e10} from \eqref{SHD-th-p-e7} and
considering \eqref{SHD-a2} and \eqref{SHD-a3} we get
\begin{eqnarray}\label{SHD-a4}
a_4 &=& \frac{5U_{1}^{2}(t)c_1(c_2-d_2)}{16(\lambda+\mu)(2\lambda+\mu)}
                +\frac{U_{1}(t)(c_3-d_3)}{4(3\lambda+\mu)}
                +\frac{(U_2(t)-U_1(t))}{4(3\lambda+\mu)}c_1(c_2+d_2)\nonumber\\
    && \qquad +\left[\frac{U_1(t)-2U_2(t)+U_3(t)}{8(3\lambda+\mu)}
                -\frac{(\mu^2+3\mu-4)U_{1}^{3}(t)}{48(\lambda+\mu)^3}\right]c_{1}^{3}.
\end{eqnarray}
Thus from \eqref{SHD-a2}, \eqref{SHD-a3} and \eqref{SHD-a4} we can easily
establish that,
\begin{eqnarray}\label{a2a4-a3}
    a_2a_4-a_3^2 & =& \frac{U_1^3(t)c_1^2(c_2-d_2)}{32(\lambda+\mu)^2(2\lambda+\mu)}
                        +\frac{U_1^2(t)c_1(c_3-d_3)}{8(\lambda+\mu)(3\lambda+\mu)}\nonumber\\
                 && \,\, +\frac{\left[U_2(t)-U_1(t)\right]U_1(t)}{8(\lambda+\mu)(3\lambda+\mu)}c_1^2(c_2+d_2)
                        -\frac{U_1^2(t)(c_2-d_2)^2}{16(2\lambda+\mu)^2}\\
                 && \,\, +\frac{U_1(t)c_1^4\left[6(U_1(t)-2U_2(t)+U_3(t))
                            (\lambda+\mu)^3-U_1^3(t)(\mu^2+3\mu+2)(3\lambda+\mu)\right]}{96(\lambda+\mu)^4(3\lambda+\mu)}~.\nonumber
\end{eqnarray}
From Lemma \ref{L-C2-C3}, we have
\begin{eqnarray}
    2c_2 & =& c_1^2+(4-c_1^2)x,  \qquad    2d_2  = d_1^2+(4-d_1^2)y
\label{CP-SHD-e17a}
\end{eqnarray}
and
\begin{eqnarray}
    4c_3 & =& c_1^3+2(4-c_1^2)c_1x-(4-c_1^2)c_1x^2+2(4-c_1^2)(1-|x|^2)z \nonumber\\
    4d_3 & =& d_1^3+2(4-d_1^2)d_1x-(4-d_1^2)d_1x^2+2(4-d_1^2)(1-|y|^2)w
\label{CP-SHD-e17b}
\end{eqnarray}
for some $x,\, y,\, z,\, w$ with $|x|\leq 1,$ $|y|\leq 1,$ $|z|\leq 1$ and $|w|\leq 1.$ Also, from \eqref{SHD-th-p-e11}, \eqref{CP-SHD-e17a} and \eqref{CP-SHD-e17b}, we obtain
\begin{eqnarray}
c_2 - d_2 = \frac{4-c_1^2}{2}(x-y), \qquad c_2 + d_2 = c_1^2 + \frac{4-c_1^2}{2}(x+y)
\label{c2=d2}
\end{eqnarray}
and
\begin{eqnarray}
c_3 - d_3 &=& \frac{c_1^2}{2} +\frac{(4-c_1^2)c_1}{2}(x+y)-  \frac{(4-c_1^2)c_1}{4}(x^2+y^2)\nonumber\\
            && \qquad +\frac{4-c_1^2}{2}[(1-|x|^2)z-(1-|y|^2)w].
\label{c3-d3}
\end{eqnarray}
According to Lemma \ref{L-C2-C3}, we may assume without any restriction
that $c\in[0,2]$, where $c_1=c.$ Using (\ref{c2=d2}) and (\ref{c3-d3}) in \eqref{a2a4-a3}, by taking
$|x|=\gamma_1$, $ |y|=\gamma_2$, we can easily obtain that,
\[
|a_2a_4-a_3^2| \leq S_1+S_2(\gamma_1+\gamma_2)+S_3(\gamma_1^2+\gamma_2^2)+S_4(\gamma_1+\gamma_2)^2=F(\gamma_1,\gamma_2),
\]
where
\begin{eqnarray*}
S_1 = S_1(c, ~t) & =&
\frac{U_1(t)|6U_3(t)(\lambda+\mu)^3-U_1^3(t)(\mu^2+3\mu+2)(3\lambda+\mu)|c^4}
                                {96(\lambda+\mu)^4(3\lambda+\mu)}\nonumber\\
                      &&\qquad +\frac{U_1^2(t)c(4-c^2)}{8(\lambda+\mu)(3\lambda+\mu)}\geq 0\\
S_2 = S_2(c, ~t) & =& \frac{U_1^3(t)c^2(4-c^2)}{64(\lambda+\mu)^2(2\lambda+\mu)}
                        +\frac{U_1(t)U_2(t)(4-c^2)c^2}{16(\lambda+\mu)(3\lambda+\mu)} \geq 0 \\
S_3 = S_3(c, ~t) & =& \frac{U_1^2(t)c(c-2)(4-c^2)}{32(\lambda+\mu)(3\lambda+\mu)}
                        \leq 0 \\
S_4 = S_4(c, ~t) & =& \frac{U_1^2(t)(4-c^2)^2}{64(2\lambda+\mu)^2} \geq 0 \qquad (\frac{1}{2} < t < 1, \, 0 \leq c \leq 2).
\end{eqnarray*}

Now we need to maximize $F(\gamma_1, \gamma_2)$ in the closed square
$$\mathbb{S}:=\{(\gamma_1,\gamma_2):\, 0\leq \gamma_1\leq 1,\, 0\leq \gamma_2 \leq 1\}.$$
\noindent Since $S_3<0$ and $S_3+2S_4>0$ for all $t \in \left(\frac{1}{2}, ~1\right)$ and $c\in (0,2),$ we conclude that
\[
F_{\gamma_1\gamma_1}F_{\gamma_2\gamma_2}-(F_{\gamma_1\gamma_2})^2 < 0 \qquad \hbox{for all} \qquad \gamma_1,\, \gamma_2 \in \mathbb{S}.
\]

Thus the function $F$ cannot have a local maximum in the
interior of the square $\mathbb{S}.$ Now, we investigate
the maximum of $F$ on the boundary of the square $\mathbb{S}.$

For $\gamma_1=0$ and $0\leq \gamma_2 \leq 1$ (similarly $\gamma_2=0$
and $0\leq \gamma_1\leq1$) we obtain
\[
F(0,\gamma_2)=G(\gamma_2)=S_1+S_2\gamma_2+(S_3+S_4)\gamma_2^2.
\]

(i) The case $S_3+S_4 \geq 0:$ In this case for $0 < \gamma_2 < 1,$
        any fixed $c$ with $0\leq c <2$ and for all $t$ with $\frac{1}{2} < t < 1,$
        it is clear that $G'(\gamma_2)=2(S_3+S_4)\gamma_2+S_2 >0,$ that is, $G(\gamma_2)$ is
        an increasing function. Hence, for fixed $c\in [0,2)$ and $t \in \left(\frac{1}{2}, ~1\right),$ the maximum of
        $G(\gamma_2)$ occurs at $\gamma_1=1$ and
        \[
            \max G(\gamma_2) = G(1) = S_1+S_2+S_3+S_4.
        \]

(ii) The case $S_3+S_4 < 0:$ Since $S_2+2(S_3+S_4)\geq 0$ for $0 < \gamma_2 <1,$
        any fixed $c$ with $0\leq c <2$ and for all $t$ with $\frac{1}{2} < t < 1,$
        it is clear that $S_2+2(S_3+S_4) < 2(S_3+S_4)\gamma_2+S_2<S_2$ and so $G'(\gamma_2)>0.$
        Hence for fixed $c\in [0,2)$ and $t \in \left(\frac{1}{2}, ~1\right),$
        the maximum of $G(\gamma_2)$ occurs at $\gamma_1=1.$

Also for $c=2$ we obtain
\begin{eqnarray}\label{F}
F(\gamma_1,\gamma_2)
&=& S_1 \left|\right._{c~=~2} \nonumber\\
& =& \frac{U_1(t)|6U_3(t)(\lambda+\mu)^3-(\mu^2+3\mu+2)U_1^3(t)(3\lambda+\mu)|}
                                {6(\lambda+\mu)^4(3\lambda+\mu)}~.
\end{eqnarray}
Taking into account the value (\ref{F}) and the cases $(i)$
and $(ii),$ for $0\leq \gamma_2 <1,$ any fixed $c$ with $0\leq c \leq 2,$
and for all $t$ with $\frac{1}{2} < t < 1,$
\[
\max G(\gamma_2) = G(1) = S_1+S_2+S_3+S_4.
\]

For $\gamma_1=1$ and $0\leq \gamma_2 \leq 1$ (similarly $\gamma_2=1$
and $0\leq \gamma_1 \leq 1$), we obtain
\[
F(1,\gamma_2)=H(\gamma_2)=(S_3+S_4)\gamma_2^2 + (S_2+2S_4)\gamma_2 + S_1 + S_2 + S_3 + S_4.
\]

Similarly, to the above cases of $S_3+S_4,$ we get that
\[
\max H(\gamma_2) = H(1) = S_1+2S_2+2S_3+4S_4.
\]

Since $G(1)\leq H(1)$ for $c \in [0,2]$ and $t \in \left(\frac{1}{2}, ~1\right),$ $\max F(\gamma_1,\gamma_2)=F(1,1)$
on the boundary of the square $\mathbb{S}.$ Thus the maximum of $F$ occurs
at $\gamma_1=1$ and $\gamma_2=1$ in the closed square $\mathbb{S}.$

Next,  let a function $K~:~[0,2] \rightarrow \mathbb{R}$ defined by
\begin{equation}\label{K}
K(c, ~t) = \max F(\gamma_1, \gamma_2) = F(1,1)= S_1+2S_2+2S_3+4S_4
\end{equation}
for fixed value of $t.$
Substituting the values of $S_1,$ $S_2,$ $S_3$ and $S_4$ in the function $K$
defined by (\ref{K}), yields
\begin{align*}
K(c, ~t) = \frac{U_1^2(t)}{(2\lambda+\mu)^2}+\frac{M_{1}c^4+12M_{2}c^2}
            {96(\lambda+\mu)^4(2\lambda+\mu)^2(3\lambda+\mu)},
\end{align*}
where
\begin{eqnarray*}\label{CP-SHD-e22}
M_{1} & =&
U_1(t)\left|6U_3(t)(\lambda+\mu)^3-(\mu^2+3\mu+2)(3\lambda+\mu)U_1^3(t)\right|(2\lambda+\mu)^2
\\ \nonumber && -3U_1(t)\left[U_1^2(t)(3\lambda+\mu)+4U_2(t)(\lambda+\mu)(2\lambda+\mu)\right](\lambda+\mu)^2(2\lambda+\mu)
\\ \nonumber && -6U_1^2(t)\lambda^2(\lambda+\mu)^3
\\
M_{2} & =& \left[U_1^3(t)(2\lambda+\mu)(3\lambda+\mu)+2(2U_2(t)+U_1(t))U_1(t)(\lambda+\mu)(2\lambda+\mu)^2
                    \right.\\ && \left.
                            -4U_1^2(t)(\lambda+\mu)^2(3\lambda+\mu)\right](\lambda+\mu)^2.
\end{eqnarray*}

Assume that $K(c, ~t)$ has a maximum value in an interior of $c\in [0,2],$  by elementary
calculation, we find that
\begin{align*}
K^{\prime}(c, ~t) = \frac{\left(M_{1}c^2+6M_{2}\right)c}{24(\lambda+\mu)^4(2\lambda+\mu)^2(3\lambda+\mu)}. \label{H2-G'}
\end{align*}

We will examine the sign of the function $K^{\prime}(c, ~t)$ depending on the different cases of
the signs of $M_{1}$ and $M_{2}$ as follows:
\begin{enumerate}
\item Let $M_{1}\geq 0$ and $M_{2}\geq 0,$ then
$K^{\prime}(c, ~t) \geq 0,$ so $K(c, ~t)$ is an increasing function. Therefore
\begin{eqnarray}
    \max\{K(c, ~t):c\in(0, ~2)\} &=& K(2^-, ~t)\nonumber \\
                             &=& \frac{U_1^2(t)}{(2\lambda+\mu)^2}+\frac{M_{1}+3M_{2}}
                                {6(\lambda+\mu)^4(2\lambda+\mu)^2(3\lambda+\mu)}.
\label{CP-SHD-e24a}
\end{eqnarray}
That is, $\max\{\max\{F(\gamma_1, ~\gamma_2) : 0 \leq \gamma_1, ~\gamma_2 \leq 1 \} ~:~ 0 < c < 2 \} = K(2^-, ~t).$
\item Let $M_{1}> 0$ and $M_{2}< 0,$ then
$c_0 = \sqrt{\frac{-6M_{2}}{M_{1}}}$ is a critical point of the function $K(c, ~t).$
We assume that, $c_0 \in (0, ~2),$ since $K^{\prime\prime}(c, ~t) > 0,$ $c_0$ is a local minimum point of the function
$K(c, ~t).$ That is the function $K(c, ~t)$ can not have a local maximum.
\item Let $M_{1}\leq 0$ and $M_{2}\leq 0,$ then
$K^{\prime}(c, ~t) \leq 0,$ so $K(c, ~t)$ is an decreasing function on the interval $(0, ~2).$ Therefore
\begin{eqnarray}
    \max\{K(c, ~t):c\in(0,2)\} &=& K(0^+, ~t) = 4S_4 =  \frac{U_1^2(t)}{(2\lambda+\mu)^2}.
\label{CP-SHD-e24b}
\end{eqnarray}
\item Let $M_{1}< 0$ and $M_{2}> 0,$ then $c_0$ is a critical point of the function
$K(c, ~t).$ We assume that $c_0 \in (0, ~2).$ Since $K^{\prime\prime}(c, ~t) < 0,$ $c_0$ is a local maximum point of the function $K(c, ~t)$
and  maximum value occurs at $c=c_0.$ Therefore
\begin{eqnarray}
    \max\{K(c, ~t):c\in(0, ~2)\} &=& K(c_0, ~t),
\label{CP-SHD-e24c}
\end{eqnarray}
where
\begin{eqnarray*}
K(c_0, ~t)   =  \frac{4t^2}{(2\lambda+\mu)^2}-\frac{3M_{2}^2}{8M_{1}(\lambda+\mu)^4(2\lambda+\mu)^2(3\lambda+\mu)}.
\end{eqnarray*}
\end{enumerate}
Thus, from \eqref{F} to \eqref{CP-SHD-e24c}, the proof of Theorem \ref{th-SHD-class} is completed.
\end{proof}

\begin{corollary}
\label{Cor1-SHD-class} Let $f\in \sigma$ of the form (\ref{Int-e1}) be in $\mathcal{B}_{\sigma }\left( \lambda ,t\right).$ Then
\[
|a_2a_4-a_3^2| \leq
        \left\{
             \begin{array}{ll}
                  K(2^-, ~t) &; M_{3}\geq0 ~\hbox{and}~ M_{4}\geq0 \\
                  max\left\{\frac{4t^2}{(2\lambda+1)^2}, ~K(2^-, ~t)\right\}
                            &; M_{3}>0 ~\hbox{and}~ M_{4}<0  \\
                  \frac{4t^2}{(2\lambda+1)^2}
                            &; M_{3}\leq0 ~\hbox{and}~ M_{4}\leq0\\
                  max \left\{K(c_0, ~t),~K(2^-, ~t)\right\}
                            &; M_{3}<0 ~\hbox{and} ~ M_{4}>0
              \end{array}
         \right.,
\]
where
\begin{eqnarray*}
    K(2^-, ~t) & =& \frac{4t^2}{(2\lambda+1)^2}
                   +\frac{M_{3}+3M_{4}}{6(\lambda+1)^4(2\lambda+1)^2(3\lambda+1)},
                   \label{CP-SHD-e2a}\\
    K(c_0, ~t) & =& \frac{4t^2}{(2\lambda+1)^2}
                   -\frac{3M_{4}^2}{8M_{3}(\lambda+1)^4(2\lambda+1)^2(3\lambda+1)},\qquad c_0 = \sqrt{\frac{-6M_{4}}{M_{3}}}
                   \label{CP-SHD-e2b}
\end{eqnarray*}
and
\begin{eqnarray*}
M_{3} & = & 16t^2 \left|3(2t^{2}-1)(\lambda+1)^3-6(3\lambda+1)t^2\right|(2\lambda+1)^2
\nonumber\\ && -24t\left[t^2(3\lambda+1)+(4t^2-1)(\lambda+1)(2\lambda+1)\right](\lambda+1)^2(2\lambda+1)
\nonumber\\ && -24t^2\lambda^2(\lambda+1)^3,
\label{CP-SHD-e3}\\
M_{4}&=& 8t\left[t^2(2\lambda+1)(3\lambda+1)
+(4t^2-1)(\lambda+1)(2\lambda+1)^2
\right. \nonumber\\ && \left.\qquad +t(2\lambda+1)^2(\lambda+1) - 2t(\lambda+1)^2(3\lambda+1)\right](\lambda+1)^2.
                            \label{CP-SHD-e4}
\end{eqnarray*}
\end{corollary}

\begin{corollary}
\label{Cor2-SHD-class} Let $f\in \sigma$ of the form (\ref{Int-e1}) be in $\mathcal{B}_{\sigma }^{\mu }\left( t\right).$ Then
\[
|a_2a_4-a_3^2| \leq
        \left\{
             \begin{array}{ll}
                  K(2^-, ~t) &; M_{5}\geq0 ~\hbox{and}~ M_{6}\geq0 \\
                  max\left\{\frac{4t^2}{(2+\mu)^2}, ~K(2^-, ~t)\right\}
                            &; M_{5}>0 ~\hbox{and}~ M_{6}<0  \\
                  \frac{4t^2}{(2+\mu)^2}
                            &; M_{5}\leq0 ~\hbox{and}~ M_{6}\leq0\\
                  max \left\{K(c_0, ~t),~K(2^-, ~t)\right\}
                            &; M_{5}<0 ~\hbox{and} ~ M_{6}>0
              \end{array}
         \right.,
\]
where
\begin{eqnarray*}
    K(2^-, ~t) & =& \frac{4t^2}{(2+\mu)^2}
                   +\frac{M_{5}+3M_{6}}{6(1+\mu)^4(2+\mu)^{2}(3+\mu)},
                   \label{CP-SHD-e2a}\\
    K(c_0, ~t) & =& \frac{4t^2}{(2+\mu)^2}
                   -\frac{3M_{6}^2}{8M_{5}(1+\mu)^4(2+\mu)^2(3+\mu)}, \qquad c_0 = \sqrt{\frac{-6M_{6}}{M_{5}}}
                   \label{CP-SHD-e2b}
\end{eqnarray*}
and
\begin{eqnarray*}
M_{5} &= & 16t^2 \left|3(2t^{2}-1)(1+\mu)^3-(\mu^2+3\mu+2)(3+\mu)t^2\right|(2+\mu)^2
\nonumber\\ && -24t\left[t^2(3+\mu)+(4t^2-1)(1+\mu)(2+\mu)\right](1+\mu)^2(2+\mu)
\nonumber\\ && -24t^2(1+\mu)^3,
\label{CP-SHD-e3}\\
M_{6}&=& 8t\left[t^2(2+\mu)(3+\mu)
+(4t^2-1)(1+\mu)(2+\mu)^2
\right. \nonumber\\ && \left.\qquad +t(2+\mu)^2(1+\mu) - 2t(1+\mu)^2(3+\mu)\right](1+\mu)^2.
                            \label{CP-SHD-e4}
\end{eqnarray*}
\end{corollary}
\begin{corollary}
\label{Cor4-SHD-class} Let $f\in \sigma$ of the form (\ref{Int-e1}) be in $\mathcal{B}_{\sigma }\left(
t\right).$ Then
\[
|a_2a_4-a_3^2| \leq \left\{
                      \begin{array}{ll}
                        t^2(1-t^2), & \hbox{$\frac{1}{2} < t \leq t_{0_{1}}$;} \\
                        \frac{t(260t^{4}+84t^{3}-139t^{2}-18t+9)}{8(18t^{3}+42t^{2}-17t-9)}, & \hbox{$t_{0_{1}} < t < 1$,}
                      \end{array}
                    \right.
\]
where, the value of $t_{0_{1}}$, which is approximately $t_{0_{1}}=0.603615$, is root of equation $M_{1}=0$ for $\lambda=\mu$ and $\frac{1}{2}<t<1.$
\end{corollary}

\begin{corollary}
\label{Cor3-SHD-class} Let $f\in \sigma$ of the form (\ref{Int-e1}) be in $\mathcal{S}_{\sigma }^{\ast
}\left( t\right).$ Then
\[
|a_2a_4-a_3^2| \leq \left\{
                      \begin{array}{ll}
                        \frac{8t^{2}}{3}, & \hbox{$\frac{1}{2} < t \leq \frac{7+\sqrt{401}}{44}$;} \\
                        t^{2}+\frac{t(2+t-11t^{2})^{2}}{3(22t^{2}-7t-4)}, & \hbox{$\frac{7+\sqrt{401}}{44}<t<1$.}
                      \end{array}
                    \right.
.
\]
\end{corollary}
\begin{remark}
For specializing the parameters involving in Theorem \ref{th-SHD-class}, the results discussed are improve the results of Mustafa \cite{Mustafa}.
\end{remark}
\
\\
\
\textbf{Acknowledgment~:~} 
We record our sincere thanks to the referees for their insightful
suggestions to improve the results as well as the present form of the article.


\end{document}